\documentclass[10pt,smallextended]{article}
\usepackage{a4wide}
\usepackage[latin1]{inputenc}
\usepackage[english]{babel}
\usepackage{geometry}
\usepackage{amsmath}
\usepackage{amssymb}
\usepackage{amsthm}
\usepackage{oldgerm}
\usepackage{amscd}
\usepackage{rotating}
\usepackage[all]{xy}
\usepackage{hyperref}
\usepackage{color}

\title{The ring of polynomials integral-valued over a finite set of integral elements}

\date{\today}

\author{Giulio Peruginelli\footnote{Department of Mathematics, University of Padova, Via Trieste, 63
35121 Padova, Italy. E-mail: gperugin@math.unipd.it}}

\numberwithin{equation}{section}

\newtheorem{Th}{Theorem}[section]
\newtheorem{Prop}[Th]{Proposition}
\newtheorem{Lemma}[Th]{Lemma}
\newtheorem{Cor}[Th]{Corollary}

\theoremstyle{definition}\newtheorem{Def}[Th]{Definition}
\newtheorem{Ex}[Th]{Example}
\newtheorem{Rem}[Th]{Remark}

\newcommand{\Q}{\mathbb{Q}}
\newcommand{\N}{\mathbb{N}}
\newcommand{\Z}{\mathbb{Z}}
\newcommand{\Int}{\mathrm{Int}}
\newcommand{\olD}{\overline{D}}
\newcommand{\olK}{\overline{K}}

\begin{document}

\leftmark{\noindent  J. Commut. Algebra 8 (2016), no. 1,  113-141.\;\;\footnotesize{\href{http://dx.doi.org/10.1216/JCA-2016-8-1-113}{http://dx.doi.org/10.1216/JCA-2016-8-1-113}}.}
{\let\newpage\relax\maketitle} 


\begin{abstract}
\noindent Let $D$ be an integral domain with quotient field $K$ and $\Omega$ a finite subset of $D$. McQuillan proved that the ring $\Int(\Omega,D)$ of polynomials in $K[X]$ which are integer-valued over $\Omega$, that is, $f\in K[X]$ such that $f(\Omega)\subset D$, is a Pr\"ufer domain if and only if $D$ is Pr\"ufer. Under the further assumption that $D$ is integrally closed, we generalize his result by considering a finite set $S$ of a $D$-algebra $A$ which is finitely generated and torsion-free as a $D$-module, and the ring $\Int_K(S,A)$ of integer-valued polynomials over $S$, that is, polynomials over $K$ whose image over $S$ is contained in $A$. We show that the integral closure of $\Int_K(S,A)$ is equal to the contraction to $K[X]$ of $\Int(\Omega_S,D_F)$, for some finite subset $\Omega_S$ of integral elements over $D$ contained in an algebraic closure $\olK$ of $K$, where $D_F$ is the integral closure of $D$ in $F=K(\Omega_S)$. Moreover, the integral closure of $\Int_K(S,A)$ is Pr\"ufer if and only if $D$ is Pr\"ufer.  The result is obtained by means of the study of pullbacks of the form $D[X]+p(X)K[X]$, where $p(X)$ is a monic non-constant polynomial over $D$: we prove that the integral closure of such a pullback is equal to the ring of polynomials over $K$ which are integral-valued over the set of roots $\Omega_p$ of $p(X)$ in $\overline K$. 
\end{abstract}

\section{Introduction}\label{intro}

Rings of integer-valued polynomials are a prominent source for providing examples of non-Noetherian Pr\"ufer domains (see the book \cite[Chapt. VI, p. 123]{CaCh}). Throughout this paper, $D$ is an integral domain which is not a field, and $K$ is its quotient field. We denote by $\overline{K}$ a fixed algebraic closure of $K$ and by $\overline{D}$ the integral closure of $D$ in $\overline{K}$. We give the following definition, which generalizes the classical definition of the ring of integer-valued polynomials over a subset (\cite[Chapt. I.1, p. 3]{CaCh}).
\vskip0.2cm
\begin{Def}
Let $R$ be an integral domain containing $D$. Let $F$ be the quotient field of $R$ (so that $K\subseteq F$). 
For a subset $\Omega$ of $F$ we set 
$$\Int_K(\Omega,R)\doteqdot \{f\in K[X] \mid f(\Omega)\subset R\},$$
which is the ring of polynomials in $K[X]$ which map every element of $\Omega$ into $R$. If $F=K$ we omit the subscript $K$. Thus, $\Int(\Omega,R)$ is a subring of $K[X]$ (the coefficients of the relevant polynomials are in the quotient field of $R$).
\end{Def}
In the case of a finite subset $\Omega$ of $D$, McQuillan studied the algebraic structure of the corresponding ring of integer-valued polynomials $\Int(\Omega,D)$, describing the spectrum of such a ring and also its additive structure (\cite{McQ}). Using McQuillan's arguments, in \cite{Boy} Boynton observed that $\Int(\Omega,D)$ fits in a pullback diagram. Here we generalize this class of rings by considering  first  a finite set $\Omega$ of integral elements in $\overline D$ and polynomials in $K[X]$ which preserve the integrality of the elements of $\Omega$, that is, for each $\alpha$ in $\Omega$, $f(\alpha)$ is integral over $D$; according to the above definition, this ring is denoted by $\Int_K(\Omega,\overline{D})$. For example, given a monic non-constant polynomial $p\in D[X]$, let $\Omega_p$ be the set of roots of $p(X)$ in a splitting field. Then the ring $\Int_K(\Omega_p,\overline{D})$ is of the above kind, and it is not difficult to show that the ring $\Int_K(\Omega,\overline{D})$, for a finite set $\Omega$ of $\overline{D}$, can be reduced to this case.  More generally, we consider a finite set $S$ of integral elements over $D$ which do not necessarily lie in an algebraic extension of $K$, i.e.: $S$ is contained in a $D$-algebra $A$, which is finitely generated and torsion-free as a $D$-module (for example, a matrix algebra or a quaternion algebra). We consider then polynomials in $K[X]$ which map the elements of $S$ into $A$: $\Int_K(S,A)=\{f\in K[X] \mid f(S)\subset A\}$. Note that $A$ is not necessarily commutative and may contain zero-divisors, and each of its elements satisfies a monic polynomial over $D$.

Given a monic polynomial $p(X)$ in $D[X]$, the study of the ring $\Int_K(\Omega_p,\overline{D})$ we are going to do goes through another kind of pullback ring. As for the rings $\Int_K(\Omega,\olD)$, the rings we introduce now are the pullbacks of the canonical residue map $K[X]\twoheadrightarrow \frac{K[X]}{p(X)K[X]}$ with respect to some subring of $\frac{K[X]}{p(X)K[X]}$, thus they are subrings of $K[X]$ sharing with $K[X]$ the ideal $p(X)K[X]$. 

\begin{Def}
Let $p(X)$ be a non-constant monic polynomial in $D[X]$. We consider the following subring of $K[X]$:
$$D(p)\doteqdot D[X]+p(X)\cdot K[X]=\{r(X)+p(X)q(X)\,|\,r\in D[X],q\in K[X]\}.$$
\end{Def}
It is straightforward to verify that the elements of this set form a ring under the usual operation of sum and product induced by the polynomial ring $K[X]$. In Lemma \ref{Dppullback} we will show that a polynomial $f(X)$ in $K[X]$ is in $D(p)$ if and only if the remainder in the division of $f(X)$ by $p(X)$ is in $D[X]$. Note that the principal ideal $p(X)\cdot K[X]$ of $K[X]$ is also an ideal of $D(p)$. We have  then the following diagram:
$$\xymatrix{
D(p)\ar@{>>}[d]\ar@{^{(}->}[r]&K[X]\ar@{>>}[d]\\
\frac{D(p)}{p(X)K[X]}\ar@{^{(}->}[r]&\frac{K[X]}{p(X)K[X]}}$$
so that $D(p)$ is a pullback of $K[X]$ (for a general reference about pullbacks see \cite{GabHou}). Examples of such pullbacks appear in \cite{Boy}, and more widely in \cite{Per2}.
 
We see at once that $D(p)$ is contained in $\Int_K(\Omega_p,\olD)$. Also, $\Int_K(\Omega_p,\olD)$ has the ideal $p(X)K[X]$ in common with $K[X]$, so, like $D(p)$, also $\Int_K(\Omega_p,\olD)$ is a pullback ring. This point of view is clearly a generalization of \cite[Example 4.4 (1)]{Boy}, which we briefly recall below in section \ref{Deg1}. 
 
We give some motivation which led us to study the pullback rings $D(p)=D[X]+p(X)\cdot K[X]$. In \cite{Per2} this kind of polynomial pullback arose as the ring of integer-valued polynomials over certain subsets of matrices. Let $M_n(D)$ be the $D$-algebra of $n\times n$ matrices with entries in $D$ and let $\Int_K(M_n(D))=\{f\in K[X]\,|\,f(M_n(D))\subset M_n(D)\}$, the ring of integer-valued polynomials over $M_n(D)$. Given a monic polynomial $p\in D[X]$ of degree $n$, we denote by $M_n^p(D)$ the set of matrices $M$ in $M_n(D)$ whose characteristic polynomial is equal to $p(X)$. We consider the overring of ${\rm Int}_K(M_n(D))$ made up by those polynomials which are integer-valued over $M_n^p(D)$, namely:
$$\Int_K(M_n^p(D),M_n(D))=\{f\in K[X]\,|\,f(M_n^p(D))\subset M_n(D)\}$$
This partition of $M_n(D)$ into subsets of matrices having prescribed characteristic polynomial was used in \cite{Per2} to give a characterization of the polynomials of $\Int_K(M_n(D))$ in terms of their divided differences (see \cite[Theorem 4.1]{Per2}). By \cite[Lemma 2.2 \& Remark 2.1]{Per2} we have 
$$\Int_K(M_n^p(D),M_n(D))=D(p).$$
In particular, the ring $\Int_K(M_n(D))$ is represented as an intersection of pullbacks $D(p)$, as $p(X)$ ranges through the set of all the monic polynomials in $D[X]$ of degree $n$ (\cite[Remarks 2.1 and 2.2]{Per2}). In \cite{PerWer} the authors address the following question, which generalizes the previous case: for a $D$-algebra $A$ as above, where $D$ is integrally closed, we consider the ring $\Int_K(A)=\{f\in K[X] \mid f(A)\subset A\}$ of integer-valued polynomials over $A$. Is $\Int_K(A)$ equal to the intersection of pullbacks of the form $D(p)$? In general, we have
$$\bigcap_{a \in A} D(\mu_a) \subseteq \Int_K(A),$$
where, for $a\in A$, $\mu_a(X)$ denotes the minimal polynomial of $a$ over $K$ (by assumption on $A$ and $D$, $\mu_a\in D[X]$ and is monic). The conditions under which the previous containment is an equality are not known.

Throughout the paper, given a monic polynomial $p(X)$ in $D[X]$, we denote by $\Omega_p$ the multi-set of its roots in $\overline{K}$ (we recall the notion of multi-set in section \ref{pullback div diff}).  

This work is organized as follows. In section \ref{pullback div diff} we recall a characterization for the polynomials in $D(p)$ in terms of their divided differences. We use this result to show that the ring $\Int^{\{n\}}(\Omega,D)$ of polynomials whose divided differences of order less than or equal to $n$ are integer-valued over a subset $\Omega$ of $D$ can be represented as an intersection of such pullbacks. This ring has been introduced by Bhargava in \cite{BG}; we recall the definition in that section. In section \ref{Int clos D(p)} we prove the following theorem:
\begin{Th}\label{1st theorem}
Let $p(X)$ a monic non-constant polynomial in $D[X]$. Then the integral closure of the ring $D(p)=D[X]+p(X)K[X]$ is the ring $\Int_K(\Omega_p,\olD)$.
\end{Th} 
As a corollary, we show that the integral closure of $\Int^{\{n\}}(\Omega,D)$ is equal to the ring $\Int(\Omega,D)$, in the case of a finite subset $\Omega$ of $D$. For a general subset $\Omega$ of $D$, in the case where $D$ has finite residue rings, an argument from \cite{PerWer} gives the same conclusion. In section \ref{PrufRingIntVal}, we prove the main theorem:
\begin{Th}\label{2nd theorem}
Assume $D$ integrally closed and let $\Omega$ be a finite subset of $\olD$. Then the ring $\Int_K(\Omega,\olD)$ is Pr\"ufer if and only if $D$ is Pr\"ufer.
\end{Th}
If $\Omega\subset D$, then this is precisely the main result obtained by McQuillan. The crucial remark is that, for a monic polynomial $p(X)$ in $D[X]$, $\Int_K(\Omega_p,\olD)\subseteq\Int_F(\Omega_p,\olD)$ is an integral ring extension, where $F=K(\Omega_p)$ is the splitting field of $p(X)$. It is not difficult to see that $\Int_F(\Omega,\olD)$ is equal to $\Int_F(\Omega,D_F)$, where $D_F$ is the integral closure of $D$ in $F$, and this is precisely the kind of ring considered by McQuillan. We note that this is a partial answer to \cite[Question 29]{PerWer}, where we asked if $\Int_K(\Omega,\olD)$ is Pr\"ufer, when $\Omega$ is a subset of integral elements of degree over $K$ bounded by some positive integer $n$. If $D$ is integrally closed, we give also a criterion to establish when the pullback $D(p)$ is integrally closed, that is, equal to $\Int_K(\Omega_p,\olD)$ (see Theorem \ref{main thm}). In particular, in the case of a Pr\"ufer domain $D$, this condition is satisfied automatically if $D(p)$ is integrally closed.

Finally, in the last section, we apply the previous results in the more general setting of a finite set $S$ of integral elements over $D$ which do not necessarily lie in an algebraic extension of $K$.

\begin{Cor}\label{3rd result} Assume $D$ integrally closed and let $S$ be a finite set of a torsion-free $D$-algebra $A$, which is finitely generated as a $D$-module. Let $\Omega_S$ be the set of roots in $\olD$ of the minimal polynomials of $s$ over $D$, as $s$ ranges through $S$. Then the integral closure of $\Int_K(S,A)$ is $\Int_K(\Omega_S,\olD)$.
\end{Cor}

\vskip0.2cm

\subsection{Preliminary results}

In the case of a monic polynomial, the following lemma determines the quotient of $D(p)$ by the ideal $p(X)K[X]$. 
We denote by $\pi:K[X]\twoheadrightarrow \frac{K[X]}{p(X)K[X]}$ the canonical residue map, which associates to a polynomial $f\in K[X]$ the residue class $f(X)+p(X)K[X]$.

\begin{Lemma}\label{Dppullback}
Let $p\in D[X]$ be a monic non-constant polynomial. Then $D(p)$ is the pullback of $\frac{D[X]}{p(X)D[X]}\hookrightarrow\frac{K[X]}{p(X)K[X]}$ with respect to the canonical residue map $\pi:K[X]\twoheadrightarrow \frac{K[X]}{p(X)K[X]}$. In other words, the following is a pullback diagram (i.e.: $D(p)=\pi^{-1}(\frac{D[X]}{p(X)D[X]})$):
$$\xymatrix{
D(p)\ar@{>>}[d]\ar@{^{(}->}[r]&K[X]\ar@{>>}[d]\\
\frac{D[X]}{p(X)D[X]}\ar@{^{(}->}[r]&\frac{K[X]}{p(X)K[X]}}$$
In particular, a polynomial $f\in K[X]$ belongs to $D(p)$ if and only if the remainder in the division by $p(X)$ in $K[X]$ belongs to $D[X]$. Equivalently, we have
$$\frac{D(p)}{p(X)\cdot K[X]}\cong \frac{D[X]}{p(X)\cdot D[X]}.$$
\end{Lemma}
\begin{proof} 
Since $p(X)$ is monic, we have two consequences. Firstly, $D[t]\cong\frac{D[X]}{p(X)D[X]}$ is a free $D$-module of rank $n=\deg(p)$ with basis $\{1,t,\ldots,t^{n-1}\}$, where $t$ is the residue class of $X$ modulo $p(X)D[X]$.  In particular, every element $r\in D[t]$ can be uniquely represented as $r(t)=\sum_{i=0,\ldots,n-1}c_i t^i$, with $c_i\in D$. Secondly, $p(X)\cdot K[X]\cap D[X]=p(X)\cdot D[X]$, so the image of the restriction of $\pi$ to $D[X]$ is isomorphic to $D[t]$. Therefore, $D[t]\cong\frac{D[X]}{p(X)D[X]}$ embeds naturally into $\frac{K[X]}{p(X)K[X]}\cong K[t]$ (the class $X \pmod{p(X)D[X]}$ is mapped to $X \pmod{p(X)K[X]}$, so without confusion we may denote them with the same letter $t$). Note that $K[t]$ is a free $K$-module of rank $n$ with the same basis $\{1,t,\ldots,t^{n-1}\}$.

We consider now the composition of mappings $D[X]\hookrightarrow D(p)\twoheadrightarrow D(p)/p(X)K[X]$. By the second consequence above and by the second isomorphism theorem we have the isomorphism of the claim. More explicitly, given $f\in K[X]$, there exist (uniquely determined) a quotient $q\in K[X]$ and a remainder $r\in K[X]$ (with either $r=0$ or $\deg(r)<\deg(p)$) such that $f(X)=r(X)+q(X)p(X)$. Hence, if $r(X)=\sum_i c_iX^i$, then  $\pi(f)=\pi(r)=r(t)=\sum_i c_i t^i\in K[t]$. From the algebraic structure of $D[t]$ we deduce that $r(t)$ is in $D[t]$ if and only if the remainder $r(X)$ is in $D[X]$. This condition in turn is equivalent to $f\in D(p)$.
\end{proof}

\vskip0.5cm

\begin{Lemma}\label{DpDq}
Let $p,q\in D[X]$ be monic polynomials. Then
\begin{center}
$D(p)$ is contained in $D(q)\Leftrightarrow p(X)$ is divisible by $q(X)$. 
\end{center}
In particular, $D(p)=D(q)\Leftrightarrow p(X)=q(X)$.
\end{Lemma}
\begin{proof} One direction is easy. Conversely, suppose $D(p)\subseteq D(q)$ so that $p(X)=r(X)+q(X)k(X)$, for some $r\in D[X]$, $r=0$ or $\deg(r)<\deg(q)$, $k\in K[X]$. If $r\not=0$, let $c\in K\setminus D$ be such that $c\cdot r(X)$ is not in $D[X]$. Then $c\cdot p$ is in $D(p)$ but it is not in $D(q)$, contradiction. Notice that $k(X)$ has to be in $D[X]$ (see also \cite[Lemma]{McA}).
\end{proof}
\vskip0.5cm

The following two cases, linear and irreducible polynomial, are given as an example and to further illustrate the connection between polynomial pullbacks and rings of integer-valued polynomials.

\subsection{Linear case}\label{Deg1}
In the linear case the connection between the polynomial pullbacks and ring of integer-valued polynomials over finite sets becomes evident. Suppose $p(X)=X-a\in D[X]$. Then the remainder of the division of a polynomial $f\in K[X]$ by $X-a$ is the value of $f(X)$ at $a$. Hence, 
$$D(p)=D+(X-a)\cdot K[X]=\Int(\{a\},D)$$
It is well-known (see for example \cite[Proposition IV.4.1]{CaCh}) that $\Int(\{a\},D)$ is integrally closed if and only if $D$ is. It is easy to see that the integral closure of $\Int(\{a\},D)$ is $\Int(\{a\},D')$, where $D'$ is the integral closure of $D$ in $K$ (notice that $\Int(\{a\},D')=D'+(X-a)K[X]$ is a pullback). More generally, we recall the following result.

\begin{Lemma}\label{IntED}
Let $E\subset K$ be a finite set.  Then the integral closure of $\Int(E,D)$ is $\Int(E,D')$, where $D'$ is the integral closure of $D$ in $K$.
\end{Lemma}
\begin{proof}
By \cite[Proposition IV.4.1]{CaCh}, 
$\Int(E,D')$ is integrally closed. Conversely, take $f\in \Int(E,D')$. Then for each $a\in E$, there exists a monic polynomial $m_{f(a)}\in D[X]$ such that $m_{f(a)}(f(a))=0$. We consider the monic polynomial of $D[X]$ equal to the product of the $m_{f(a)}(X)$'s, as $a$ ranges through $E$. Then $m(f(X))$ is in $\Int(E,D)$, because for each $a\in E$ we have $m(f(a))=0\in D$. This gives a monic equation for $f(X)$ over $\Int(E,D)$. \end{proof}
\vskip0.4cm
\begin{Rem}\label{IntED pullback}
We recall now the following observation made in \cite{Boy}. Under the assumptions of Lemma \ref{IntED}, $\Int(E,D)$ is the pullback of $\prod_{i=1}^m D\subset \prod_{i=1}^m K$ with respect to the canonical mapping $\pi:K[X]\twoheadrightarrow \frac{K[X]}{p(X)K[X]}\cong\prod_{i=1}^m K$, where $p(X)=\prod_{a\in E}(X-a)$. The map $\pi$ is given by $f(X)\mapsto(f(a))_{a\in E}$. Notice also that $p(X)K[X]$ is an ideal of $\Int(E,D)$, because every polynomial of $K[X]$ which is divisible by $p(X)$ is zero on $E$. In particular, we have the following isomorphism of $D$-modules
\begin{equation*} 
\frac{\Int(E,D)}{p(X)K[X]}\cong\prod_{i=1}^m D
\end{equation*}
\end{Rem}

\subsection{Irreducible polynomial case}\label{irrpol}

We suppose now that $D$ is integrally closed and $p(X)$ is a monic irreducible polynomial in $D[X]$ of degree $n>0$. It is easy to see (see for example \cite{McA} or \cite[Proposition 11, Chapt. V]{Bourbaki}) that $p(X)$ is irreducible in $K[X]$, so that $p\in D[X]$ is also prime  and $D[X]/(p(X))\cong D[\alpha]$, where $\alpha$ is a root of $p(X)$ in $\overline{K}$. The next proposition follows by \cite[Prop. 3.1]{Per1} (which is proved in the case $D=\Z$). We sketch the proof for the sake of the reader, giving emphasis to the relevant points.

\begin{Prop}\label{intcloirr} Let $p\in D[X]$ be a monic and irreducible polynomial, with set of roots $\Omega_p\subset \overline{K}$. Let $F=K(\Omega_p)$ be the splitting field of $p(X)$ over $K$ and $D_F$ the integral closure of $D$ in $F$. For each $\alpha\in \Omega_p$ we set 
$$S_\alpha\doteqdot\Int_K(\{\alpha\},D_\alpha),$$
where $D_{\alpha}$ is the integral closure of $D$ in $K(\alpha)\subseteq F$.

Then, for each $\alpha\in \Omega_p$, $S_{\alpha}=\Int_K(\Omega_p,D_F)$ and this ring is the integral closure of $D(p)$. Moreover, $D(p)$ is integrally closed if and only if $D_{\alpha}=D[\alpha]$, for some (hence all) $\alpha\in\Omega_p$.
\end{Prop}
\begin{proof} Using a Galois-invariance argument it is easy to show that the ring $S_\alpha$ does not depend on the choice of the root $\alpha$ of $p(X)$ and is equal to $\Int_K(\Omega_p,D_F)$. We observe that $S_\alpha=\{f\in K[X]\,|\,f(\alpha)\textnormal{ is integral over }D\}$. Then for a polynomial $f\in S_\alpha$ and for every conjugate $\alpha'$ of $\alpha$ over $K$, $f(\alpha')$ is integral over $D$ as well. Since $D[\alpha]$, for $\alpha\in\Omega_p$, is a free $D$-module of rank $n$, we can show that 
$$D(p)=\{f\in K[X] \mid f(\alpha)\in D[\alpha]\}=\Int_K(\{\alpha\},D[\alpha]).$$
Finally, using a pullback diagram argument, since $D_{\alpha}$ is the integral closure of $D[\alpha]$ in $K(\alpha)$, we deduce that $\Int_K(\Omega_p,D_F)$ is the integral closure of $D(p)$ (see \cite[Proposition 3.1]{Per1} for the details). \end{proof}
\vskip0.2cm 
In particular, the proposition shows that all the subrings $\Int(\{\alpha\},D_F)\subset F[X]$, for $\alpha\in\Omega_p$, contracts in $K[X]$ to the same ring $S_\alpha$. Notice also that $\Int_K(\Omega_p,\olD)$ is equal to $\Int_K(\Omega_p,D_F)$, where $D_F$ is the integral closure of $D$ in the splitting field $F=K(\Omega_p)$ of $p(X)$ over $K$.

\vskip0.4cm
\section{Pullbacks and divided differences}\label{pullback div diff}

In this section we recall a result of \cite{Per2} which characterizes a polynomial $f(X)$ in a pullback $D(p)=D[X]+p(X)\cdot K[X]$ in terms of a finite set of conditions on the evaluation of the divided differences of $f(X)$ at the roots of $p(X)$ in $\overline{K}$. We use this result to show that the ring of integer-valued polynomials whose divided differences are also integer-valued can be represented as an intersection of such pullbacks. 

Given a polynomial $f\in K[X]$, the divided differences of $f(X)$ are defined recursively as follows: 
\begin{align*}
\Phi^{0}(f)(X_0)\doteqdot& f(X_0)\\
\Phi^{1}(f)(X_0,X_1)\doteqdot& \frac{f(X_0)-f(X_1)}{X_0-X_1}\\
\ldots\\
\Phi^{k}(f)(X_0,\ldots,X_k)\doteqdot&\frac{\Phi^{k-1}(f)(X_0,\ldots,X_{k-1})-\Phi^{k-1}(f)(X_0,\ldots,X_{k-2},X_k)}{X_{k-1}-X_{k}}\\
\end{align*}
For each $k\in\N$, $\Phi^{k}(f)$ is a symmetric polynomial over $K$ in $k+1$ variables (see \cite{EvFaJoh}, \cite{Per2}, \cite{Steff} and \cite{vdW1} for the main properties of the divided differences of a polynomial). We recall here that, given a finite sequence of elements $a_0,\ldots,a_n$ of a commutative ring $R$, and a polynomial $f\in R[X]$ of degree $\leq n$ we have the following expansion due to Newton:
\begin{align}\label{Newt}
f(X)=f(a_0)+\Phi^1(f)(a_0,a_1)(X-a_0)+\Phi^2(f)(a_0,a_1,a_2)(X-a_0)(X-a_1)+\ldots\nonumber\\
+\Phi^{n}(f)(a_0,\ldots,a_n)(X-a_0)\cdot\ldots\cdot(X-a_{n-1})
\end{align}
Since in general a polynomial may not have distinct roots, we need to recall the following definition.
\begin{Def}
A \textbf{multi-set} is a collection of elements $\Omega$ in which elements may occur multiple times. The number of times an element occurs is called its multiplicity in the multi-set. The cardinality of a multi-set $\Omega$ is defined as the number of elements of $\Omega$, each of them counted with multiplicity. The underlying set of $\Omega$ is the (proper) set containing the distinct elements in $\Omega$.\\
A multi-set $\Omega_1$ is a sub-multi-set of a multi-set $\Omega_2$ if every element $\alpha$ of $\Omega_1$ of multiplicity $n_1$ belongs to $\Omega_2$ with multiplicity $n_2\geq n_1$. 
\end{Def}

\begin{Rem}
Let $\Omega$ be a multi-set of cardinality $n$ and let $S$ be the underlying set of $\Omega$. The choice of an ordering on the elements of $\Omega$ corresponds to a $n$-tuple in $S^n$ (we have thus $n!$ choices). Conversely, given an $n$-tuple $\underline{s}$ in $S^n$, where $S$ is a set, if we do not consider the order its components, we have a multi-set $\Omega$ of cardinality $n$.\end{Rem}

\begin{Rem}\label{order div diff} A particular ring of integer-valued polynomials involving divided differences has been introduced by Bhargava in \cite{BG}. Given a subset $S$ of $D$ and $n\in\N$, we consider those polynomials $f(X)$ in $K[X]$ whose $k$-th divided difference $\Phi^k(f)$ is integer-valued on $S$ for all $k\in\{0,\ldots,n\}$, namely: 
$$\Int^{\{n\}}(S,D)\doteqdot\{f\in K[X] \mid \forall 0\leq  k\leq n,\,\Phi^k(f)(S^{k+1})\subset D\}.$$
For $n=0$ we recover the ring $\Int(S,D)$, which contains $\Int^{\{n\}}(S,D)$ for all $n\in\N$.

Given $f\in\Int^{\{n\}}(S,D)$ and $k\in\{0,\ldots,n\}$ we have:
\begin{equation}\label{*}\tag{*}
\forall (a_1,\ldots,a_{k+1})\in S^{k+1},\; \Phi^k(f)(a_1,\ldots,a_{k+1})\in D.
\end{equation}
Since $\Phi^k(f)$ is a symmetric polynomial in $k+1$ variables, for all permutations $\sigma
\in\mathcal{S}_{k+1}$ we have $\Phi^k(f)(a_1,\ldots,a_{k+1})=\Phi^k(f)(a_{\sigma(1)},\ldots,a_{\sigma(k+1)})$. Hence, we may disregard the order of the components of the chosen $(k+1)$-tuple. If we consider a multi-set $\Omega$ of cardinality $k+1$ formed by elements of $S$, we may define 
$\Phi^k(f)(\Omega)$ as the value of $\Phi^k(f)$ at one of the $(k+1)$-tuple associated to $\Omega$. Thus we choose an ordering of $\Omega$ and, by above, the value $\Phi^k(f)(\Omega)$ does not depend on the chosen ordering. Notice that $\Omega$ is not necessarily a sub-multi-set of $S$. We only require that the underlying set of $\Omega$ is contained in $S$. For example, if $S=\{1,2,3\}$ and $k=1$, we have $\{1,1\}$, $\{1,3\}$ and $\{2,2\}$ as possible choices for $\Omega$. 

We may rephrase the above property (\ref{*}) by saying that for all multi-sets $\Omega$ of cardinality $k+1$ such that the underlying set $\Omega'$ is contained in $S$, we have $\Phi^k(f)(\Omega)\in D$.  
\end{Rem}
\vskip0.2cm
\noindent\textsc{Notation.}
We fix now the notation for the rest of this section.
\begin{itemize}
 \item[-] $p(X)$ is a monic non-constant polynomial in $D[X]$ of degree $n$.
 \item[-] $\Omega_p=\{\alpha_1,\ldots,\alpha_n\}$ is the multi-set of roots of $p(X)$ in $\overline{K}$ (the $\alpha_i$'s are integral over $D$).
 \item[-] $F=K(\alpha_1,\ldots,\alpha_n)$ the splitting field of $p(X)$.
 \item[-] $D_F$ the integral closure of $D$ in $F$.
\end{itemize}
\vskip0.3cm
Given $f\in F[X]$, whenever we expand $f\in F[X]$ as in (\ref{Newt}) in terms of the roots $\Omega_p$ of $p(X)$, we implicitly assume that an order of $\Omega_p$ has been fixed (so we choose one of the $n!$ associated $n$-tuples). Changing the order of $\Omega_p$ will give a different expansion.

We need now the following preliminary lemma: the divided differences of a polynomial $p(X)$ are zero when they are evaluated at the roots of the polynomial $p(X)$ itself.

\begin{Lemma}\label{Phiproots}
For every sub-multi-set $\Omega$ of $\Omega_p$ of cardinality $k+1$, $k<n-1$, we have $\Phi^k(p)(\Omega)=0$, and $\Phi^{n-1}(p)(\Omega_p)=1$. Equivalently, we have:
$$\Phi^k(p)(\alpha_1,\ldots,\alpha_{k+1})=\left\{
\begin{array}{ccl}
0,&{\rm if   }& 0\leq k<n\\
1,&{\rm if   }& k=n
\end{array}\right.$$
for any possible choice of an ordering for $\Omega_p$.
\end{Lemma}
\begin{proof} We fix an ordering for $\Omega_p$. We consider the Newton expansion of $p(X)$ over $F$ with respect to $\Omega_p$ up to the order $n$ ($p(X)$ is split over $F$). The coefficients of this expansion are exactly $\{\Phi^k(p)(\alpha_1,\ldots,\alpha_{k+1})\}_{0\leq k\leq n}$, where for $k=n$ we have the leading coefficient of $p(X)$ which is $1$. Since $p(X)$ is divisible by itself, all the other coefficients in this expansion are zero. Obviously, the result does not depend on the chosen ordering for $\Omega_p$. \end{proof}

\begin{Lemma}\label{Phifr}
Let $f\in K[X]$ and let $r\in K[X]$ be the unique remainder in the division of $f(X)$ by $p(X)$ in $K[X]$. If $r\not=0$, let $m<n$ be the degree of $r(X)$. Then over $F[X]$ we have
\begin{equation}\label{restoNewt}
r(X)=f(\alpha_1)+\Phi^1(f)(\alpha_1,\alpha_2)\cdot (X-\alpha_1)+\ldots+\Phi^{m}(f)(\alpha_1,\ldots,\alpha_{m+1})\prod_{i=1}^{m}(X-\alpha_i)
\end{equation}
which is the Newton expansion of $r(X)$ with respect to $\Omega_p=\{\alpha_1,\dots,\alpha_n\}$.
\end{Lemma}
\begin{proof} If $f(X)=q(X)p(X)+r(X)$, by linearity of the divided difference operator, we have $\Phi^k(f)=\Phi^k(r)+\Phi^k(p\cdot q)$, for all $k\in\N$. Moreover, by the so-called Leibniz rule for divided differences (see for example \cite{Steff}), we have $\Phi^k(p\cdot q)=\sum_{i=0,\ldots,k}\Phi^i(p)\Phi^{k-i}(q)$ (we omit the variables). By Lemma \ref{Phiproots}, for $0\leq k<n$, we get that
\begin{equation}\label{Phifp}
\Phi^k(f)(\alpha_{1},\ldots,\alpha_{k+1})=\Phi^k(r)(\alpha_{1},\ldots,\alpha_{k+1}).
\end{equation}
Notice that, for $k=m$ the above value is the leading coefficient of $r(X)$, and for $m<k<n$ it is zero. Because of the last formula, $r(X)$ has the desired expansion over $F[X]$. \end{proof}

By means of Lemma \ref{Phiproots} and Lemma \ref{Phifr} we give a new proof of \cite[Proposition 4.1]{Per2}, which says that a polynomial $f(X)$ of $K[X]$ is in $D(p)$ if and only if the divided differences of $f(X)$ up to the order $n-1$ are integral on every sub-multi-set of the multi-set $\Omega_p$ of the roots of $p(X)$. 

\begin{Prop}\label{charpullback}
Let $D$ be an integrally closed domain with quotient field $K$. Let $f\in K[X]$ and $p\in D[X]$ monic of degree $n$. Let $\Omega_p=\{\alpha_1,\ldots,\alpha_n\}$ be the multi-set of roots of $p(X)$ in a splitting field $F$ over $K$. Then the following are equivalent:
\begin{itemize}
 \item[i)] $f\in D(p)$.
 \item[ii)] for all $0\leq k<n$, $\Phi^k(f)(\alpha_{1},\ldots,\alpha_{k+1})\in D[\alpha_{1},\ldots,\alpha_{k+1}]$.
 \item[iii)] for all $0\leq k<n$, $\Phi^k(f)(\alpha_{1},\ldots,\alpha_{k+1})\in D_F$.
\end{itemize}
\end{Prop}
\begin{proof} If i) holds, let $f(X)=r(X)+p(X)q(X)$, for some $q\in K[X]$, $r\in D[X]$, $\deg(r)<n$ or $r=0$. In particular, the divided differences of $r(X)$ are polynomials with coefficients in $D$. By (\ref{Phifp}), $\Phi^k(f)(\alpha_{1},\ldots,\alpha_{k+1})=\Phi^k(r)(\alpha_{1},\ldots,\alpha_{k+1})\in D[\alpha_{1},\ldots,\alpha_{k+1}]$, for all the relevant $k$'s. Hence, i)$\Rightarrow$ ii).

Obviously ii)$\Rightarrow$ iii), since the roots of $p(X)$ are integral over $D$, so that $D[\alpha_{1},\ldots,\alpha_{k+1}]\subseteq D_F$.

Suppose now that iii) holds. We have to prove that the remainder $r(X)$ of the Euclidean division in $K[X]$ of $f(X)$ by $p(X)$ is in $D[X]$. Let $m<n$ be the degree of $r(X)$. Consider the Newton expansion of $r(X)$ with respect to $\Omega_p$ over $F[X]$ as in Lemma \ref{Phifr} (see (\ref{restoNewt})). By assumption, the coefficients $\{\Phi^k(f)(\alpha_{1},\ldots,\alpha_{k+1})\}_{k=0,\ldots,m}$ of this expansion are in $D_F$. The leading coefficient of $r(X)$ is equal to $\Phi^m(f)(\alpha_{1},\ldots,\alpha_{m+1})$, so that it is in $D_F\cap K=D$ (we use here the assumption that $D$ is integrally closed).  The coefficient $c_{m-1}$ of the term $X^{m-1}$ of $r(X)$ is $\Phi^{m-1}(f)(\alpha_1,\ldots,\alpha_{m})\pm(\sum_{i=1,\ldots,m}\alpha_i)\cdot\Phi^{m}(f)(\alpha_1,\ldots,\alpha_{m+1})$ which is in $D_F$, so $c_{m-1}$ is in $K\cap D_F=D$. If we continue in this way we prove that $r(X)$ is in $D[X]$, which gives i). \end{proof}

\begin{Rem}\label{pullback div diff submulti}
If we choose another ordering on the multi-set $\Omega_p$ of roots of $p(X)$ we have other conditions of integrality on the values of the divided differences of a polynomial $f\in D(p)$ at the vectors of elements in $\Omega_p$. Since condition i) of Proposition \ref{charpullback} does not depend on the order we choose on $\Omega_p$, the above conditions are also equivalent to this one: 
\begin{center}
ii') for all $0\leq k<n$, and for every sub-multi-set $\Omega$ of $\Omega_p$ of cardinality $k+1$, $\Phi^k(f)(\Omega)\in D[\Omega]$,
\end{center}
that is, $\Phi^k(f)$ is integral-valued on $\Omega$: $\Phi^k(f)(\Omega)\in D_F$ (see also \cite[Proposition 4.1 \& Remark 4.1]{Per2}).

Note that, if $p\in D[X]$ is a monic polynomial of degree $n$ which is split over $D$, that is, $p(X)=\prod_{i=1}^n(X-a_i)$, $a_i\in D$, then condition i) and ii) are equivalent without the assumption that $D$ is integrally closed (this follows immediately from the formula (\ref{restoNewt})). In particular, condition ii) becomes: for all $0\leq k<n$, $\Phi^k(f)(a_{1},\ldots,a_{k+1})\in D$. We have thus in this case found again the result of \cite[Proposition 11]{EvFaJoh} (see also \cite[Lemma 2.2 \& Remark 2.1]{Per2}).
\end{Rem}

Now we give the link between the ring of integer-valued polynomials whose divided differences are also integer-valued introduced by Bhargava and the polynomial pullbacks $D(p)$ we are working with. 

We observe first that, if $p\in D[X]$ is a monic polynomial of degree $n$ which is split over $D$ (i.e.: the set of roots $\Omega_p$ is contained in $D$), then $\Int^{\{n-1\}}(\Omega_p,D_F)$ may be strictly contained in $D(p)$.  

\begin{Ex}\label{Esempio}
Let $n=2$, $\Omega=\{1,3\}\subset\Z$ and $p(X)=(X-1)(X-3)$. Let $f(X)=p(X)/3\in \Z(p)$. We have that $\Phi^1(f)(1,1)=-2/3$, so that $f\notin\Int^{\{1\}}(\Omega,\Z)$.\\
Indeed, by Proposition \ref{charpullback}, given any $f\in \Z(p)$, $\Phi^1(f)$ is integer-valued over $\{(1,3),(3,1)\}\subsetneq \Omega^2$.
\end{Ex}
We need to introduce another notation before the next theorem.
\vskip0.3cm
\noindent\textsc{Notation.}
Let $\Omega$ be a subset of $D$ and let $n$ be a positive integer. We denote by $\mathcal{P}_n(\Omega)$ the set of monic polynomials $q(X)$ over $D$ of degree $n$ whose set of roots is contained in $\Omega$ (so, in particular, they are split over $D$).

\begin{Th}\label{Int div diff pullback}
Let $\Omega\subseteq D$ and $n\in\N$. Then
$$\Int^{\{n-1\}}(\Omega,D)=\bigcap_{q\in \mathcal{P}_n(\Omega)}D(q).$$
\end{Th}
\begin{proof} $(\subseteq)$. Let $f\in \Int^{\{n-1\}}(\Omega,D)$ and let $q\in \mathcal{P}_n(\Omega)$. Since for all $0\leq k<n$ we have $\Phi^k(f)(\Omega^{k+1})\subset D$, then for each sub-multi-set $\{a_1,\ldots,a_{k+1}\}$ of  $\Omega_q$ of cardinality $k+1$  we have $\Phi^k(f)(a_1,\ldots,a_{k+1})\in D$. Then by Proposition \ref{charpullback} (see also Remark \ref{pullback div diff submulti}), we have that $f\in D(q)$.

$(\supseteq)$. Let $f\in D(q)$, for all $q\in \mathcal{P}_n(\Omega)$. Let $k\in\{0,\ldots,n-1\}$ and let $(a_1,\ldots,a_{k+1})\in\Omega^{k+1}$. We consider a polynomial $q\in\mathcal{P}_n(\Omega)$ such that the multi-set $\{a_1,\ldots,a_{k+1}\}$ is a sub-multi-set of the multi-set of roots $\Omega_q$ (that is, $\prod_{i=1}^{k+1}(X-a_i)$ divides $q(X)$). Then by Proposition \ref{charpullback}, condition ii), $\Phi^k(f)(a_1,\ldots,a_{k+1})\in D$ (see also Remark \ref{pullback div diff submulti}, condition ii')). Since $(a_1,\ldots,a_{k+1})$ was chosen arbitrarily, $f(X)$ is in $\Int^{\{n-1\}}(\Omega,D)$.
\end{proof}
\vskip0.2cm
In the Example \ref{Esempio} above, we have that $f(X)=\frac{q(X)}{3}-\frac{2}{3}(X-1)$ is not in $\Z(q)$, where $q(X)=(X-1)^2$ is a polynomial in $\mathcal{P}_2(\Omega)=\{(X-1)(X-3),(X-1)^2,(X-3)^2\}$. 
\vskip0.5cm
\begin{Rem}\label{remark on Intdivdiff intersection pullbacks}
By \cite[Lemma 5.1]{Per2}, given a monic polynomial $p\in D[X]$ of degree $n$ which is split over $D$, the pullback ring $D(p)$ is equal to $\Int_K(T_n^p(D),M_n(D))$, where $T_n^p(D)$ is the set of $n\times n$ triangular matrices with characteristic polynomial equal to $p(X)$. In particular, we have this representation for the ring of integer-valued polynomials over the algebra of $n\times n$ triangular matrices over $D$:
\begin{equation}\label{IntTnD pullbacks}
\Int_K(T_n(D))=\bigcap_{p\in\mathcal{P}_n^s(D)}D(p)
\end{equation}
where $\mathcal{P}_n^s(D)$ is the set of monic polynomials over $D$ of degree $n$ which are split over $D$; as we mentioned in the introduction, a similar result holds for $\Int_K(M_n(D))$, see \cite{Per2}. We note that this gives a positive answer to \cite[Question 31]{PerWer} for the algebra $T_n(D)$. Similarly, given any subset $\mathcal{P}$ of $\mathcal{P}_n^s(D)$, the intersection of the pullbacks $D(p)$ as $p(X)$ ranges through $\mathcal{P}$ is the ring of polynomials which are integer-valued over the set of triangular matrices whose characteristic polynomial belongs to $\mathcal{P}$. By Theorem \ref{Int div diff pullback}, this ring is equal to $\Int^{\{n-1\}}(\Omega,D)$, where $\Omega\subseteq D$ is the set of roots of the polynomials in $\mathcal{P}$.

In the case $\Omega=D$, \cite[Theorem 16]{EvFaJoh} proves that $\Int^{\{n-1\}}(D)=\Int_K(T_n(D))$, which by (\ref{IntTnD pullbacks}) is also equal to the intersection of the pullbacks $D(p)$, as $p(X)$ ranges through $\mathcal{P}_n^s(D)$. Therefore, Theorem \ref{Int div diff pullback} generalizes this result to any subset $\Omega$ of $D$.
\end{Rem}

\section{Integral closure of polynomial pullbacks}\label{Int clos D(p)}
\vskip0.2cm
\begin{Rem}\label{IntKOmegaD integ closed}
Let $\Omega\subset\overline{K}$ be a finite set. Let $F=K(\Omega)$ and let $D_F$ be the integral closure of $D$ in $F$. By \cite[Proposition IV.4.1]{CaCh}, $\Int(\Omega,D_F)$ is integrally closed. Hence, $\Int_K(\Omega,D_F)=\Int(\Omega,D_F)\cap K[X]$ is integrally closed, too. The same remark was used in \cite[Proposition 7]{PerWer}.
\end{Rem}

\begin{Lemma}\label{pPp}
Let $D$ be an integrally closed domain. Let $p\in D[X]$ be a non-constant polynomial of degree $n$ and $\Omega_p\subset\overline{K}$ the multi-set of its roots. Let $f\in K[X]$ be integral-valued over $\Omega_p$, that is, $f\in\Int_K(\Omega_p,\overline D)$. Then the polynomial
$$P(X)=P_{f,p}(X)\doteqdot\prod_{\alpha\in\Omega_p}(X-f(\alpha))$$
is in $D[X]$.  Moreover, $P(f(X))$ is divisible by $p(X)$ in $K[X]$.
\end{Lemma}
\begin{proof} Notice that $P(X)$ has degree $n$, because the product is over the elements of the multi-set $\Omega_p$. We set $g(X)\doteqdot\frac{p(X)}{{\rm lc}(p)}=\prod_{\alpha\in\Omega_p}(X-\alpha)$, where ${\rm lc}(p)$ is the leading coefficient of $p(X)$. The polynomial $g(X)$ is in $K[X]$ and is monic.

Let $M\in M_n(K)$ be a matrix with characteristic polynomial equal to $g(X)$ (e.g., the companion matrix of $g(X)$). The multi-set of eigenvalues of $M$ over $\overline{K}$ is exactly $\Omega_p$. Notice that $f(M)$ is in $M_n(K)$, so its characteristic polynomial is in $K[X]$. By \cite[Chap. VII, Proposition 10]{BourbakiAlg}  (considering everything over $\overline{K}$) the characteristic polynomial of $f(M)$ is precisely $P(X)$. In particular, the set of eigenvalues of $f(M)$ is $f(\Omega_p)=\{f(\alpha)\,|\,\alpha\in\Omega_p\}$, which, by assumption on $f(X)$, is contained in $\olD$. Hence, the coefficients of $P(X)$ are integral over $D$ (being the elementary symmetric functions of the roots), and since $D$ is integrally closed they are in $D$. 

For the last statement, notice that for each $\alpha\in\Omega_p$, $X-\alpha$ divides $f(X)-f(\alpha)$ over $F=K(\Omega_p)$. Hence, $p(X)=\prod_{\alpha\in\Omega_p}(X-\alpha)$ divides $P(f(X))=\prod_{\alpha\in\Omega_p}(f(X)-f(\alpha))$ over $F$. Since both polynomials are in $K[X]$, one divides the other over $K$, as we wanted. \end{proof}
\vskip0.5cm
We prove now Theorem \ref{1st theorem} of the Introduction. For the sake of the reader we repeat here the statement.
\begin{Th}\label{Intclopull}
Let $p\in D[X]$ be a monic non-constant polynomial and let $\Omega_p\subset\overline{K}$ be the multi-set of its roots. Then the integral closure of $D(p)$ is $\Int_K(\Omega_p,\olD)$.
\end{Th}
\noindent Notice that by definition $\Int_K(\Omega_p,\overline D)=\Int_K(\Omega_p',\overline D)$, where $\Omega_p'$ is the underlying set of $\Omega_p$, the set of distinct roots of $p(X)$. 

\begin{proof} Remember that $\Int_K(\Omega_p,\olD)$ is integrally closed by the Remark at the beginning of this section. If $D'$ is the integral closure of $D$ in its quotient field $K$, then $D(p)\subseteq D'(p)$ is an integral ring extension, because $D[X]\subseteq D'[X]$ is. 
Since $D(p)\subseteq D'(p)\subseteq \Int_K(\Omega_p,\olD)$ (because $p(X)$ is monic, so $\Omega_p$ is contained in $\olD$), without loss of generality we may assume that $D$ is integrally closed (that is, $D=D'$). To prove the statement, it suffices to prove that $D(p)\subseteq \Int_K(\Omega_p,\olD)$ is an integral ring extension.

Let $f\in \Int_K(\Omega_p,\olD)$ and consider $P(X)$ defined as in Lemma \ref{pPp}. Then $P(X)$ is a monic polynomial in $D[X]$ such that $P(f(X))$ is divisible by $p(X)$ over $K$. Hence, $P(f(X))$ is in $D(p)$, and this gives a monic integral equation for $f(X)$ over the pullback ring $D(p)$. \end{proof}
\vskip0.3cm

We prove now that the ring of polynomials in $K[X]$ whose divided differences of order up to $n$ are integer-valued over a finite subset $\Omega$ of $D$ has integral closure equal to the ring of polynomials which are integer-valued over $\Omega$.

\begin{Cor}\label{Intclo div diff}
Let $D$ be an integrally closed domain. Let $\Omega\subset D$ be a finite set and let $n\in \N$. Then the integral closure of $\Int^{\{n\}}(\Omega,D)$ is $\Int(\Omega,D)$.
\end{Cor}
\begin{proof} As in the proof of Theorem \ref{Intclopull}, it is sufficient to show that any $f\in\Int(\Omega,D)$ satisfies a monic equation over the ring $\Int^{\{n\}}(\Omega,D)$.

By Theorem \ref{Int div diff pullback}, $\Int^{\{n\}}(\Omega,D)$ is equal to the intersection of the pullbacks $D(p)$, as $p(X)$ ranges through the finite family $\mathcal{P}_{n+1}(\Omega)$ of monic polynomials over $D$ of degree $n+1$ whose set of roots is contained in $\Omega$. We consider the subset $\mathcal{P}$ of $\mathcal{P}_{n+1}(\Omega)$ of those polynomials of the form $q(X)=(X-a)^{n+1}$, for $a\in\Omega$. For each of them we consider the polynomial $P_{f,q}\in D[X]$ as defined in Lemma \ref{pPp}. Therefore
$$Q(X)\doteqdot\prod_{q\in\mathcal{P}}P_{f,q}(X)$$
is a monic polynomial in $D[X]$ such that $Q(f(X))$ is in $p(X)K[X]$ for each $p\in \mathcal{P}_{n+1}(\Omega)$. In fact, 
let $p\in \mathcal{P}_{n+1}(\Omega)$. If $a\in\Omega$ is a root of $p(X)$ of multiplicity $e\leq \deg(p)=n+1$, then $(X-a)^e$ divides $(f(X)-f(a))^{n+1}$ over $K$. Notice that the latter is a factor of $Q(f(X))$. Since this holds for every root of $p(X)$, then $p(X)$ divides $Q(f(X))$ over $K$, that is, $Q(f(X))\in pK[X]\subset D(p)$. Since this holds for every $p\in \mathcal{P}_{n+1}(\Omega)$, this concludes the proof of the Corollary. \end{proof} 

\begin{Rem}
If $\Omega\subseteq D$ is an infinite set and $D$ has finite residue rings (that is, $D/dD$ is a finite ring for every non-zero $d\in D$), reasoning as in \cite{PerWer} by means of the pullback representation of $\Int^{\{n\}}(\Omega,D)$ given by Theorem \ref{Int div diff pullback}, the same result of Corollary \ref{Intclo div diff} holds. For $\Omega=D$, the result was given in \cite[Corollary 17]{PerWer}, where it is proved that the integral closure of $\Int_K(T_{n+1}(D))$ is $\Int(D)$. Note that, by \cite[Theorem 16]{EvFaJoh}, the former ring is equal to $\Int^{\{n\}}(D)$ (see Remark \ref{remark on Intdivdiff intersection pullbacks}).
\end{Rem}

\section{Pr\"ufer rings of integral-valued polynomials}\label{PrufRingIntVal}
\vskip0.2cm
The next lemma, though easy, is a crucial step to establish when $\Int_K(\Omega,\olD)$ is a Pr\"ufer domain, for a finite set $\Omega$ of integral elements over $D$.
\begin{Lemma}\label{integral extension pullbacks}
Let $p\in D[X]$ be a monic non-constant polynomial and $K\subseteq F$ be an algebraic extension. Let $D_F$ the integral closure of $D$ in $F$. Then $D(p)\subseteq D_F(p)$ is an integral ring extension.
\end{Lemma}
\begin{proof} We use the well-known fact that the integral closure of $D[X]$ in $F[X]$ is $D_F[X]$ (\cite[Proposition 13, Chapt. V]{Bourbaki}). Hence, given $f(X)=r(X)+p(X)q(X)\in D_F(p)$, for some $r\in D_F[X]$ ($r=0$ or $\deg(r)<\deg(p)$) and $q\in F[X]$, the polynomial $r(X)$ is integral over $D[X]$, so in particular it is also integral also over $D(p)$. We show now that $h(X)=p(X)q(X)\in p(X)\cdot F[X]$ is integral over $D(p)$.

It is easy to see that, if $\Psi_q(T,X)$ is the minimal polynomial of $q(X)$ over $K[X]$, then the minimal polynomial of $h(X)$ over $K[X]$ is given by $\Psi_{h}(T,X)=p^n\cdot\Psi_{q}(\frac{T}{p},X)$, which is a monic polynomial in $T$ over $D$. Notice that the coefficients of  $\Psi_{h}(T,X)-T^n$ are in $p(X)\cdot K[X]$, so that $\Psi_{h}(T,X)\in D(p)[T]$. This proves our assertion. \end{proof}
\vskip0.2cm
We prove now Theorem \ref{2nd theorem} of the Introduction. 
\vskip0.2cm
\begin{Th}\label{Prufer condition}
Assume that $D$ is integrally closed and let $\Omega$ be a finite subset of $\olD$. Then $\Int_K(\Omega,\olD)$ is Pr\"ufer if and only if $D$ is Pr\"ufer.

\begin{proof} Given $f(X)$ in $\Int_K(\Omega,\overline D)$ and $\alpha\in\Omega$, $f(X)$ is integral-valued over all the conjugates of $\alpha$ over $K$ (see Proposition \ref{intcloirr}). Hence, without loss of generality, we can assume that $\Omega$ is equal to the set of roots $\Omega_p$ of a monic polynomial $p(X)$ over $D$ (more precisely, $p(X)$ is the product of all the minimal polynomials of the elements of $\Omega$, without repetitions).

Let $F=K(\Omega_p)$ be the splitting field of $p(X)$ over $D$ and let $D_F$ be the integral closure of $D$ in $F$. By assumption, $\Omega_p\subset D_F$. Remember that $\Int_K(\Omega_p,\overline D)=\Int_K(\Omega_p,D_F)$ (see the remarks after Proposition \ref{intcloirr}). By the result of McQuillan (\cite[Corollary 7]{McQ}), $\Int(\Omega_p,D_F)$ is a Pr\"ufer domain if and only if $D_F$ is a Pr\"ufer domain. Since $D$ is integrally closed, by \cite[Theorem 22.3 \& 22.4]{GilmBook} $D$ is Pr\"ufer if and only if $D_F$ is Pr\"ufer. We have the following diagram:
$$\xymatrix{
D_F(p)\ar[r]&\Int(\Omega_p,D_F)\ar[r]&F[X]\\
D(p)\ar[u]\ar[r]&\Int_K(\Omega_p,D_F)\ar[u]\ar[r]&K[X]\ar[u]
}$$
By Theorem \ref{Intclopull}, $D(p)\subseteq\Int_K(\Omega_p,D_F)$ and $D_F(p)\subseteq\Int(\Omega_p,D_F)$ are integral ring extensions. Hence, by Lemma \ref{integral extension pullbacks}, $\Int(\Omega_p,D_F)$ is integral over $\Int_K(\Omega_p,D_F)$. Moreover, since the former ring is integrally closed, it is the integral closure of the latter ring in $F[X]$.
Finally, we have these equivalences:
\begin{center}
$D$ Pr\"ufer $\Leftrightarrow$ $D_F$ Pr\"ufer $\Leftrightarrow \Int(\Omega_p,D_F)$ Pr\"ufer $\Leftrightarrow \Int_K(\Omega_p,D_F)$ Pr\"ufer
\end{center}
where the last equivalence follows again by \cite[Theorem 22.3 \& 22.4]{GilmBook} ($\Int_K(\Omega_p,D_F)$ is integrally closed by Remark \ref{IntKOmegaD integ closed}). 
\end{proof}
\end{Th}

As we recalled in the introduction, the intersection of the polynomial pullbacks $D(p)$ arises in many different contexts, especially those concerning rings of integer-valued polynomials over algebras. In section \ref{pullback div diff} we saw that the ring of integer-valued polynomials whose divided differences are also integer-valued can be represented as an intersection of such pullbacks.  We now investigate more deeply how these pullbacks intersect with each other. As a corollary, we obtain a criterion for a pullback $D(p)$ to be integrally closed. 

At the beginning of Section \ref{irrpol} we recalled that a monic irreducible polynomial over an integrally closed domain $D$ is still irreducible over the quotient field $K$. Moreover, a monic polynomial $p\in D[X]$ can be uniquely factored into monic irreducible polynomials over $D$ (see \cite{McA}; this is a sort of Gauss' Lemma for monic polynomials over an integrally closed domain). Therefore, given a monic polynomial $p(X)$ in $D[X]$, we have $p(X)=\prod_i q_i(X)$, where $q_i(X)$ are powers of monic irreducible polynomials in $D[X]$. In particular, the $q_i(X)$'s are pairwise coprime in $K[X]$ (but they may not be coprime over $D$, see below). A polynomial $p(X)$ is square-free exactly when each $q_i(X)$ is irreducible. Notice that $p(X)K[X]$ is an ideal of each pullback $D(q_i)$, for all $i$. In particular, it is an ideal of the intersection of the rings $D(q_i)$.

The next proposition is a generalization of Lemma \ref{Dppullback}. Recall that two ideals $I,J$ of a commutative ring $R$ are coprime if $I+J=R$ (see \cite[Chapt. 2, p. 53]{Bourbaki}).  For this statement we do not require $D$ to be integrally closed. Given $q_1,q_2\in D[X]$, we simply say that $q_1(X)$ and $q_2(X)$ are coprime (over $D$) if the corresponding principal ideals $q_1(X)D[X]$ and $q_2(X)D[X]$ are coprime. 

\begin{Prop}\label{quotient intersection pullbacks}
 
Let $p\in D[X]$ be a monic polynomial. Let $p(X)=\prod_{i}q_i(X)$  be a factorization into monic polynomials over $D$ which are pairwise coprime when they are considered over $K$. Then 
$$\frac{\bigcap_{i}D(q_i)}{p(X)K[X]}\cong \prod_{i}\frac{D[X]}{q_i(X)D[X]}.$$
Moreover, $D(p)=\bigcap_i D(q_i)$ if and only if $\{q_i(X)\}_i$ are pairwise coprime over $D$.

\end{Prop}

Note that two polynomials $q_1,q_2\in D[X]$ may be coprime over $K$ without being coprime over $D$: for example, $q_1(X)=X$ and $q_2(X)=X-2$ over $\Z$. However, under this condition, it is easy to verify that $q_1(X)D[X]\cap q_2(X)D[X]=q_1(X)q_2(X)D[X]$.

\begin{proof} It is sufficient to notice that $\bigcap_{i}D(q_i)$ is the pullback of
$$\prod_{i}\frac{D[X]}{q_i(X)D[X]}\subset \prod_{i}\frac{K[X]}{q_i(X)K[X]}\cong \frac{K[X]}{p(X)K[X]}$$
with respect to the canonical residue mapping 
$$\pi:K[X]\twoheadrightarrow \frac{K[X]}{p(X)K[X]},$$
that is: 
$$\pi^{-1}(\prod_{i}\frac{D[X]}{q_i(X)D[X]})=\bigcap_{i}D(q_i).$$

Indeed, by definition we have
$$\pi^{-1}\left(\prod_{i}\frac{D[X]}{q_i(X)D[X]}\right)=\{f\in K[X] \mid f \pmod{q_i(X)K[X]}\in \frac{D[X]}{q_i(X)D[X]},\forall i\}.$$
Since each $q_i(X)$ is monic, by Lemma \ref{Dppullback} this is equivalent to the fact that the remainder of the division of $f(X)$ by $q_i(X)$ is in $D[X]$, that is, $f(X)$ is in $D(q_i)$, hence the statement regarding the isomorphism. We have then the following pullback diagram:
$$\xymatrix{
D(p)\ar@{>>}[d]\ar@{^{(}->}[r]&\bigcap_i D(q_i)\ar@{>>}[d]\\
\frac{D[X]}{p(X)D[X]}\ar@{^{(}->}[r]&\prod_i\frac{D[X]}{q_i(X)D[X]}}$$
where the vertical arrows are the quotient map modulo the common ideal $p(X)K[X]$.  Note that the bottom horizontal arrow is injective by the remark above before the proof.  Then $D(p)=\bigcap_i D(q_i)$ if and only if $D[X]/p(X)D[X]$ and $\prod_i D[X]/q_i(X)D[X]$ are isomorphic. By the converse of the Chinese Remainder Theorem (see \cite[Chapt. 2, \S 1, Proposition 5]{Bourbaki}) this holds if and only if the principal ideals $q_i(X)D[X]$ are pairwise coprime.
\end{proof}

Recall that, given two polynomials $p_1,p_2\in D[X]$, the principal ideals $p_i(X)D[X]$, $i=1,2$, are coprime if and only if the resultant ${\rm Res}(p_1,p_2)$ is a unit of $D$ if and only if $p_1,p_2$ have no common root modulo any maximal ideal $M\subset D$. Notice that $p_1(X),p_2(X)$ are coprime in $K[X]$ if and only if ${\rm Res}(p_1,p_2)\not=0$.

The next proposition is a generalization of Remark \ref{IntED pullback}: given a monic non-constant square-free polynomial $p(X)$ in $D[X]$, we determine the quotient ring of $\Int_K(\Omega_p,\olD)$ modulo the principal ideal $p(X)K[X]$. Note that in the case $\Omega_p\subset D$, we have $\Int_K(\Omega_p,\olD)=\Int(\Omega_p,D)$ and we are in the case already treated (essentially by McQuillan).

\begin{Prop}\label{Quotient IntKOmegap}
Let $p\in D[X]$ be a monic non-constant polynomial which is square-free, say $p(X)=\prod_{i=1}^k p_i(X)$, where $p_i(X)$, for $i=1,\ldots,k$, are monic, distinct and irreducible polynomials over $D$. Then
$$\frac{\Int_K(\Omega_p,\olD)}{p(X)K[X]}\cong \prod_{i=1}^k D_{K_i}$$
where $D_{K_i}$ is the integral closure of $D$ in the field $K_i\cong\frac{K[X]}{p_i(X)K[X]}$, for each $i=1,\ldots,k$.
\begin{proof}
For each $i=1,\ldots,k$, we set 
$K_i\doteqdot K[X]/(p_i(X))\cong K[\alpha_i]$, which is a finite field extension of $K$, where $\alpha_i$ is a (fixed) root of $p_i(X)$. Let also $D_{K_i}$ be the integral closure of $D$ in $K_i$, for $i=1,\ldots,k$. By assumption on the $p_i(X)$'s, $D[X]/(p_i(X)D[X])\cong D[\alpha_i]\subset K[\alpha_i]$. Note that $\Int_K(\Omega_p,\olD)=\Int_K(\{\alpha_1,\ldots,\alpha_k\},\olD)$: if $f\in K[X]$ is integral-valued on $\alpha_i$ then it is integral-valued on every conjugate root of $\alpha$ of $\alpha_i$, that is on the set of roots $\Omega_{p_i}$ (see also Proposition \ref{intcloirr}).

As we remarked in the introduction, the rings $\Int_K(\Omega_p,\olD)\subset K[X]$ have the ideal $p(X)K[X]$ in common, so that $\Int_K(\Omega_p,\olD)$ is a pullback with respect to the canonical residue map $\pi:K[X]\twoheadrightarrow \frac{K[X]}{p(X)K[X]}$. The polynomial ring $K[X]$ is mapped to $K[X]/(p(X))\cong\prod_{i=1}^k K[\alpha_i]$ by the map which sends $X$ to $(\alpha_1,\ldots,\alpha_k)$, so that a polynomial $f\in K[X]$ is mapped to $(f(\alpha_1),\ldots,f(\alpha_k))$.

In the same way as in Proposition \ref{quotient intersection pullbacks} we have just to prove that $\Int_K(\Omega_p,\olD)=\pi^{-1}(\prod_{i=1}^k D_{K_i})$. By definition,
$$\pi^{-1}\left(\prod_{i=1}^k D_{K_i}\right)=\{f\in K[X] \mid f(\alpha_i)\in D_{K_i},\forall i=1,\ldots,k\}$$
so that a polynomial $f(X)$ is in this ring if and only if it is integral-valued on every $\alpha_i$, that is, $f\in \Int_K(\{\alpha_1,\ldots,\alpha_k\},\olD)$.
\end{proof}
\end{Prop}
An equivalent statement of Proposition \ref{Quotient IntKOmegap} is the following: let $\Omega$ be a finite subset of $\olD$ and let $p\in D[X]$ be the product of the minimal polynomials $p_1(X),\ldots,p_k(X)$ of the elements  of $\Omega$, without repetitions. Then the quotient of $\Int_K(\Omega,\olD)$ modulo $p(X)K[X]$ is isomorphic to $\prod_{i=1}^k D_{K_i}$, where $D_{K_i}$ is as in the statement of Proposition \ref{Quotient IntKOmegap}. We notice that a proof of Theorem \ref{Prufer condition} follows also in another way by \cite[Theorem 4.3]{Boy}, due to Proposition \ref{Quotient IntKOmegap}. 

\vskip0.5cm
\begin{Th} \label{main thm} Let $p\in D[X]$ be a monic non-constant  polynomial. Suppose that $p(X)=\prod_{k=1,\ldots,k}p_i(X)^{e_i}$ is the unique factorization of $p(X)$ into powers of monic irreducible polynomials in $D[X]$, $e_i\geq 1$. Then $D(p)$ is integrally closed if and only if the following conditions are satisfied:
\begin{itemize}
 \item[i)] $p(X)$ is squarefree (i.e.: $e_i=1$ for all $i$).
 \item[ii)] for each $i=1,\ldots,k$, $D[X]/(p_i(X))\cong D_{K_i}$, where the latter is the integral closure of $D$ in the field $K_i\cong K[X]/(p_i(X))$.
 \item[iii)] ${\rm Res}(p_i,p_j)\in D^*$ for each $i\not=j$.
\end{itemize} 
If $D$ is a Pr\"ufer domain, $D(p)$ is integrally closed if and only if $D(p)$ is a Pr\"ufer domain, and in that case $D(p)=\Int_K(\Omega_p,\olD)$.
\end{Th}

\begin{proof} Suppose that $D(p)$ is integrally closed. If $p(X)$ is not squarefree, then some exponent $e_i$ is strictly greater than $1$. Let $q(X)=\prod_{i=1}^k p_i(X)$ be the square-free part of $p(X)$. By assumption, $q(X)\not=p(X)$ and $q(X)$ divides $p(X)$. So by Lemma  \ref{DpDq}, $D(p)\subsetneq D(q)$. Since $q(X)$ has the same set of roots of $p(X)$, $D(q)$ is contained in $\Int_K(\Omega_p,D_F)$. Hence, $D(p)$ cannot be equal to $\Int_K(\Omega_p,D_F)$ which is in contradiction with Theorem \ref{Intclopull}. Then $p(X)$ is square-free. 

By Propositions \ref{quotient intersection pullbacks} and \ref{Quotient IntKOmegap} (we retain the same notation of those Propositions) we have the following diagram of pullbacks (notice that $\Omega_p=\bigcup_{i=1}^k \Omega_{p_i}$ and $\bigcap_{i=1}^k\Int_K(\Omega_{p_i},\olD)=\Int_K(\Omega_p,\olD)$), where the vertical lines are the reduction map modulo $p(X)K[X]$:
\begin{equation}\label{pullback diagram}
\xymatrix{
D(p)\ar@{>>}[d]\ar@{^{(}->}[r]&\bigcap_i D(p_i)\ar@{>>}[d]\ar@{^{(}->}[r]&\Int_K(\Omega_p,\olD)\ar@{>>}[d]\ar@{^{(}->}[r]&K[X]\ar@{>>}[d]\\
\frac{D[X]}{p(X)D[X]}\ar@{^{(}->}[r]&\prod_i D[\alpha_i]\ar@{^{(}->}[r]&\prod_i D_{K_i}\ar@{^{(}->}[r]&\prod_i K[\alpha_i]
}
\end{equation}
Obviously, $D(p)$ is integrally closed if and only if $D(p)=\bigcap_{i=1}^k D(p_i)$ and $\bigcap_{i=1}^k D(p_i)=\Int_K(\Omega_p,\olD)$.

Since $D(p)=\bigcap_{i=1}^k D(p_i)$, by Proposition \ref{quotient intersection pullbacks} this condition is equivalent to condition iii). Looking at the above diagram, $\bigcap_{i=1}^k D(p_i)=\Int_K(\Omega_p,\olD)$ if and only if $\frac{D[X]}{p_i(X)D[X]}= D_{K_i}$ for all $i=1,\ldots,k$, which is condition ii). 

Conversely, suppose conditions i), ii) and iii) hold. Then looking at the above pullback diagram again, we have that $D(p)$ is equal to $\Int_K(\Omega_p,\olD)$, hence, by Theorem \ref{Intclopull}, $D(p)$ is integrally closed.

Suppose now $D$ is a Pr\"ufer domain. If $D(p)=\Int_K(\Omega_p,\olD)$ then $D(p)$ is a Pr\"ufer domain by Theorem \ref{Prufer condition}. Conversely, if $D(p)$ is Pr\"ufer then it is integrally closed. The very last assertion follows at once by Theorem \ref{Intclopull}. 
\end{proof}
\vskip0.5cm
In the next examples we show that the theorem does not hold if we remove one of the conditions.

\begin{Ex} Let $p_1(X)=X^2+1, p_2(X)=X^2-2\in\Z[X]$ and $p(X)=p_1(X)p_2(X)$. The resultant ${\rm Res}(p_1,p_2)$ is equal to $9$. Moreover, $K_1=\Q(i)\supset O_{K_1}=\Z[X]/(p_1(X))$ and $K_2=\Q(\sqrt{2})\supset O_{K_2}=\Z[X]/(p_2(X))$.\\
Then $\Z(p_1)\cap\Z(p_2)={\rm Int}_\Q(\Omega_p,\overline{\Z})$ (see the proof of Theorem \ref{main thm} and the diagram (\ref{pullback diagram})) but $\Z(p_1\cdot p_2)=\Z(p)\subsetneq \Z(p_1)\cap\Z(p_2)$ (Proposition \ref{quotient intersection pullbacks}). Notice that $\Z(p_1),\Z(p_2)$ are integrally closed: $\Z(p_i)={\rm Int}_\Q(\Omega_{p_i},\overline{\Z})$ for $i=1,2$, but $\Z(p)$ is not integrally closed. Here, condition iii) of Theorem \ref{main thm} is not satisfied.
\end{Ex}

\begin{Ex} $p_1(X)=X^2-5$, $p_2(X)=X^2-6$. The resultant of $p_1(X)$ and $p_2(X)$ is equal to $1$. Then $K_1=\Q(\sqrt{5})\supset O_{K_1}=\Z[\frac{1+\sqrt{5}}{2}]\supsetneq\Z[\sqrt{5}]\cong\Z[X]/(p_1(X))$ and $K_2=\Q(\sqrt{6})\supset O_{K_2}=\Z[X]/(p_2(X))$.\\
Then $\Z(p_1)\cap\Z(p_2)\subsetneq {\rm Int}_\Q(\Omega_p,\overline{\Z})$ but $\Z(p)=\Z(p_1)\cap\Z(p_2)$. Hence, $\Z(p)$ is not integrally closed, because condition ii) of Theorem \ref{main thm} is not satisfied.
\end{Ex}

\begin{Cor}
Let $p\in D[X]$ be a monic polynomial over $D$ which is split in $D$. Then $D(p)$ is integrally closed if and only if the discriminant of $p(X)$ is a unit in $D$.
\end{Cor}
\noindent Notice that if the latter condition holds, in particular $p(X)$ is separable, that is, it has no repeated roots. We denote by $\Delta(p)$ the discriminant of $p(X)$.
\begin{proof} Let $\Omega_p=\{\alpha_1,\ldots,\alpha_n\}\subset D$ be the multi-set of roots of $p(X)$. By Theorem \ref{Intclopull} the integral closure of $D(p)$ is ${\rm Int}(\Omega_p,D)=\bigcap_i D(X-\alpha_i)$. 

It is enough to observe that $\Delta(p)=\prod_{i<j}(\alpha_i-\alpha_j)^2$ and that if $p_i(X)=X-\alpha_i$, for $i=1,\ldots,n$, then ${\rm Res}(p_i,p_j)=\pm(\alpha_j-\alpha_i)$. Then by Theorem \ref{main thm} we conclude. \end{proof}
\begin{Rem} 
The statement is false if we do not assume that $p(X)$ is split over $D$. For example, let $D=\Z$ and $p(X)=X^2-2$. Then $\Z(p)$ is integrally closed by Proposition \ref{intcloirr} (see also Theorem \ref{main thm}), because $\Z[\sqrt{2}]$ is the ring of integers $O_K$ of $K=\Q(\sqrt{2})$, so $\Z(p)={\rm Int}_{\Q}(\{\pm\sqrt{2}\},O_K)$.   However, $\Delta(p)=8$. This implies that the pullback $O_K(p)\subset K[X]$ is not integrally closed: the polynomial $f(X)=\frac{X-\sqrt{2}}{2\sqrt{2}}$ is in ${\rm Int}(\{\pm\sqrt{2}\},O_K)$ and not in $O_K(p)$, and by Theorem \ref{Intclopull} $f(X)$ is integral over $O_K(p)$.
\end{Rem}

\begin{Rem}
We can prove Theorem \ref{Intclopull} by means of a pullback diagram argument. 
By Lemma \ref{Dppullback} and Proposition \ref{Quotient IntKOmegap}, looking at the diagram (\ref{pullback diagram}), by \cite[Lemma 1.1.4 (8)]{FonHuckPap}, $\Int_K(\Omega_p,\olD)$ is the integral closure of $D(p)$, since $\prod_i D_{K_i}$ is the integral closure of $\frac{D[X]}{p(X)D[X]}$ in $\frac{K[X]}{p(X)K[X]}$. Indeed, it is known that $\frac{K[X]}{p(X)K[X]}$ is the total quotient ring of $\frac{D[X]}{p(X)D[X]}$ (see the proof of \cite[Theorem 10.15]{Nag}). Hence, by \cite[Proposition 2.7]{GilmBook}, $\frac{K[X]}{p(X)K[X]}$ is the total quotient ring of every subring containing $\frac{D[X]}{p(X)D[X]}$, and in particular of $\prod_i D[\alpha_i]$. By \cite[Proposition 9, Chapt. V]{Bourbaki}, $\prod_{i=1}^k D_{K_i}$ is the integral closure of $D$ in $\prod_{i=1}^k K[\alpha_i]$. Since each $\alpha_i$ is integral over $D$, it follows that $\prod_{i=1}^k D_{K_i}$ is the integral closure of $\prod_{i=1}^k D[\alpha_i]$ in $\prod_{i=1}^k K[\alpha_i]$. 
\end{Rem}
\vskip0.6cm

\section{General case of a finite set of integral elements over $D$}

We show in this section how to apply the previous results to the more general setting mentioned in the introduction, namely when the finite set of integral elements over $D$ is not necessarily contained in an algebraic extension of $K$. We recall the assumptions we mentioned in the introduction.

For simplicity, we assume that $D$ is integrally closed. Let $A$ be a $D$-algebra, possibly non-commutative and with zero-divisors, which is finitely generated and torsion-free as a $D$-module. Note that every element $a$ of $A$ is integral over $D$. Let $\mu_a(X)$ be the minimal polynomial of $a$ over $D$, which is not necessarily irreducible. To be precise, $\mu_a(X)$ is the monic generator of the ideal of $K[X]$ of those polynomials which are zero on $a$. Since $D$ is supposed integrally closed and $a$ is integral over $D$, $\mu_a\in D[X]$ (so that $\mu_a(X)$ is also the generator of the ideal of $D[X]$ of those polynomials which are zero at $a$). For short, we set $\Omega_a=\Omega_{\mu_a}$, the set of roots in $\olD$ of $\mu_a(X)$. We may evaluate polynomials of $K[X]$ at the elements of $A$ in the extended $K$-algebra $B=A\otimes_D K$ (note that, by assumption, $K$ and $A$ embed into $B$). Given a subset $S$ of $A$, we consider the ring of integer-valued polynomials over $S$:
$$\Int_K(S,A)=\{f\in K[X] \mid f(S)\subset A\}.$$
For $S=A$, we have the ring $\Int_K(A,A)=\Int_K(A)$ of integer-valued polynomials over $A$. For more details about this setting we refer to \cite{PerWer}. As in \cite{PerWer}, we consider polynomials over $K$ whose evaluation at the elements of $S$ are not necessarily in $A$, but are still integral over $D$. For this reason, we call them \emph{integral-valued} polynomials over $S$, since they preserve the integrality of the elements of $S$. We retain the notation introduced in \cite{PerWer}.
\begin{Def}
Let $K[S]$ be the $K$-subalgebra of $B=A\otimes_D K$ generated by $K$ and the elements of $S$. Let also $S'$ be the subset of $K[S]$ of those elements which are integral over $D$. We set
$$\Int_K(S,S')=\{f\in K[X] \mid f(S)\subset S'\}$$
which we call integral-valued polynomials over $S$.
\end{Def}
Note that in general $S'$ does not form a ring, if $A$ is non-commutative (even if $S$ is a ring; for example, consider the case $A=M_n(D)$). Nevertheless,  $\Int_K(S,S')$ does form a ring by the argument given in \cite[Proposition 6]{PerWer}: in order to show that $\Int_K(S,S')$ is closed under addition and multiplication, it is sufficient to consider what happens point-wise and use the fact that for each $s\in S$, $K[s]$ is a commutative $K$-algebra. 
We note that the ring $\Int_K(S,S')$ is equal to the ring of polynomials in $K[X]$ such that $f(s)$ (which a priori is in $K[s]\subseteq B$) is integral over $D$ for each $s\in S$. Clearly, $\Int_K(S,A)\subseteq\Int_K(S,S')$, because every element of $A$ is integral over $D$. The key result which links the ring of integral-valued polynomials $\Int_K(S,S')$ to a previous ring of integral-valued polynomials over a subset $\Omega$ of $\olD$ is the following.
\begin{Th}{\cite[Theorem 9]{PerWer}} Let $S$ be a subset of $A$ and set $\Omega_S=\bigcup_{s\in S}\Omega_s\subset\olD$. Then
$$\Int_K(S,S')=\Int_K(\Omega_S,\olD).$$
\end{Th}
\begin{proof}
For the sake of the reader we give the proof. Since $1\in D\subset B$, we may embed $B$ into the endomorphism ring $\textnormal{End}_K(B)$, via the map given by  multiplication on the left by $b\in B$. In particular,  $A$ is a sub-$D$-algebra of $\textnormal{End}_K(B)$, and for $s\in S$, $\Omega_s$ is the set of eigenvalues (in $\olK$) of $s$ considered as a $K$-endomorphism of $B$.
Since $A$ is finitely generated as a $D$-module, by \cite[Chapt. VII, \S. 5, Proposition 10]{BourbakiAlg}, for any polynomial $f\in K[X]$, $f(\Omega_s)=\{f(\alpha) \mid \alpha\in\Omega_s\}$ is the set of eigenvalues of $f(s)$, so that in our notation $f(\Omega_s)=\Omega_{f(s)}$. Given $f\in K[X]$ and $s\in S$, $f(s)$ is integral over $D$ if and only if the elements of $\Omega_{f(s)}=f(\Omega_s)$ are integral over $D$ (because $D$ is integrally closed). The claim is then proved.
\end{proof}
\vskip0.4cm
We are ready to give the proof of the last main result of the paper, see Corollary \ref{3rd result} of the Introduction.
\begin{Cor}
Let $S$ be a finite subset of $A$ and $\Omega_S=\bigcup_{s\in S}\Omega_s\subset\olD$. Then the integral closure of $\Int_K(S,A)$ is $\Int_K(\Omega_S,\olD)$.
\end{Cor}
\begin{proof}
Let $p(X)=\prod_{s\in S}\mu_s(X)\in D[X]$. By above, we have the following inclusions:
$$D(p)\subseteq\Int_K(S,A)\subseteq\Int_K(S,S')=\Int_K(\Omega_S,\olD)$$
and the claim follows by Theorem \ref{Intclopull}. 
\end{proof}

Note that by Theorem \ref{Prufer condition}, the ring  $\Int_K(S,A)$ has Pr\"ufer integral closure if and only if $D$ is Pr\"ufer.

\vskip1cm
\noindent {\bf Acknowledgements}.\\
We wish to thank the referee for his/her valuable suggestions which improved the quality of the paper. The author was supported by the Austrian Science Foundation (FWF), Project Number P23245-N18, INDAM and also University of Roma 3. The paper was prepared during a visit at the Department of Mathematics and Physics of the University of Roma 3. The author wishes to thanks Francesca Tartarone for the hospitality.

\end{document}